\documentclass[a4paper,11pt,english,leqno]{amsart}
\usepackage[utf8x]{inputenc}
\usepackage{pgf,tikz,amssymb}
\usepackage[colorlinks=true]{hyperref}
\usepackage[alphabetic,backrefs]{amsrefs}
\usetikzlibrary{arrows}
\theoremstyle{plain}
\newtheorem{thm}{Theorem}[section]

\newtheorem{lem}[thm]{Lemma}

\newtheorem{prop}[thm]{Proposition}

\newtheorem{cor}[thm]{Corollary}

\theoremstyle{definition}
\newtheorem{defn}[thm]{Definition}

\theoremstyle{remark}
\newtheorem*{dem}{Proof}

\theoremstyle{plain}

\def\S#1{\mathbb{S}^{#1}}
\def\Real#1{\Re(#1)}
\def\Esp#1#2{\text{\rm E}_{#1}\left[#2\right]}

\def\E#1{\mathcal{E}_{#1}}
\def\X{\mathfrak{X}_{*}^{(N)}}

\def\R{\mathbb{R}}
\def\C{\mathbb{C}}

\begin{document}

\title{A generalization of the spherical ensemble to even-dimensional spheres}

\author{Carlos Beltr\'{a}n and Uju\'{e} Etayo}

\thanks{This research has been partially supported by Ministerio de Economía y Competitividad, Gobierno de España, through grants MTM2017-83816-P,  MTM2017-90682-REDT, and by the Banco de Santander and Universidad de Cantabria grant 21.SI01.64658.}

\date{\today{}}



\address{Departamento de Matem\'aticas, Estad\'{\i}stica y Computaci\'on, 
Universidad de Cantabria,  Fac. Ciencias, Avd. Los Castros s/n, 39005 Santander, Spain}
\email{etayomu@unican.es}

\address{Departamento de Matem\'aticas, Estad\'{\i}stica y Computaci\'on, 
Universidad de Cantabria,  Fac. Ciencias, Avd. Los Castros s/n, 39005 Santander, Spain}
\email{beltranc@unican.es}

\begin{abstract}

In a recent article \cite{EJP3733},  Alishahi and Zamani  discuss  the spherical ensemble,  a rotationally invariant determinantal point process on $\mathbb{S}^{2}$. 
In this paper we extend this process in a natural way to the $2d$--dimensional sphere $\mathbb{S}^{2d}$. We prove that the expected value of the Riesz s-energy associated to this determinantal point process has a reasonably low value compared to the known asymptotic expansion of the minimal Riesz s-energy.

\end{abstract}

\maketitle

\tableofcontents


\section{Introduction}

Determinantal point processes (DPPs) are becoming a standard tool for generating random points on a set $\mathbb{X}$. One of the main properties of these point processes is that their statistics can be described in terms of a kernel $K(x,y), x,y \in \mathbb{X}$, which generally turns out to be the reproducing kernel of a finite dimensional subspace of $L^{2}(\mathbb{X})$. The complete theory of DPPs can be found in the excellent book \cite{Hough_zerosof}; see also \cite{BE2018} for a brief (yet, sufficient for many purposes) introduction and explanation of the main concepts.

We are interested in using DPPs for generating points in the sphere $\mathbb{S}^{d}$ that are well distributed in some sense.
For this aim, it is essential to find subspaces of $L^{2}(\mathbb{S}^{d})$ whose kernels preserve the properties of the structure of the sphere.
In \cite{BMOC2015energy} a DPP using spherical harmonics (i.e. associated to the subspace of $L^{2}(\mathbb{S}^{d})$ given by the span of bounded degree real--valued spherical harmonics) is described. 
The random configurations coming from that point process are called the {\em harmonic ensemble}, that turns out to be optimal in the sense that it minimizes Riesz 2-energy (see Sec. \ref{section_energy}) among a certain class of DPPs obtained from subspaces of real--valued functions defined in $\S{d}$.

However in the very special case of the sphere of dimension two, there exists another DPP, the so--called {\em spherical ensemble}, that corresponds to a subspace of $L^2(\mathbb{S}^2)$ coming from a weighted space of polynomials in the complex plane. The spherical ensemble produces low--energy configurations that are indeed better than those of the harmonic ensemble, see \cite{EJP3733}. A fundamental property of the spherical ensemble proved in \cite{krishnapur2009} is that it is equivalent to computing the generalized eigenvalues of pairs of matrices $(A,B)$ with complex Standard Gaussian entries. Generalized eigenvalues live naturally in the complex projective space $\mathbb{P}(\C^2)$ which is by the Riemann sphere model equivalent to the $2$--sphere.

In \cite{BE2018} a generalization of the spherical ensemble to the general projective space $\mathbb{P}(\C^d)$ was presented and called the {\em projective ensemble}. Its natural lift to the odd dimensional sphere in $\C^{d+1}\equiv\R^{2d+2}$ was proved to have lower energy (for a family of energies) than those coming from the harmonic ensemble.

In this paper we generalize the spherical ensemble to the case of spheres of even dimension. 
We are also able to compute a bound for the expected value of the Riesz s-energy of a set of $N$ points coming for this generalization.
In order to obtain this bound we prove a inequality regarding the incomplete beta function that we haven't found in the literature.

The structure of the paper is as follows. 
In section \ref{sec:kernel} we discuss the properties that a reproducing kernel  on the sphere might have.
In section \ref{sec:iran} we present briefly the spherical ensemble in the $2-$dimensional sphere.
In section \ref{sec:2d} we describe our generalization to the $2d-$dimensional sphere.
We state our main result bounding the Riesz s-energy of this DPP in section \ref{section_energy} and in section \ref{sec:beta} we prove an inequality regarding the incomplete beta function. 
Finally, in section \ref{sec_proofs} we give the proofs of the theorems.


\section{Homogeneous vs isotropic kernels}
\label{sec:kernel}


The symmetries of the sphere suggests what type of properties the ``good kernels'' should satisfy.
As a final goal, we would like the kernels to be invariant under the isometry group of the sphere, but weaker statements could also be useful.
A lot of adjectives describing kernels can be found in the literature; we now state our terminology in aims of clarity.
 
\begin{defn}\label{defn:kernel:isotropic}
A DPP of $N$ points on $\mathbb{S}^{d}$ has \textit{isotropic} associated reproducing kernel $K(p,q)$ if there exists a function $f:[0,2] \longrightarrow \mathbb{R}$ such that
\begin{equation*}
|K(p,q)| = f(||p-q||)
\end{equation*}
for all $p,q \in \mathbb{S}^{d}$.
\end{defn} 

When the kernel is isotropic, we say that the DPP is rotationally invariant.
A weaker property will be that of homogeneus intensity.
 
\begin{defn}\label{defn:kernel:homogeneus}
A DPP of $N$ points on $\mathbb{S}^{d}$ has \textit{homogeneous} associated reproducing kernel if $K(p,p)$ is constant for all $p \in \mathbb{S}^{d}$.
\end{defn}  
 
Actually, the value of this constant is determined by the volume of the sphere:
\begin{equation}\label{eq:volS}
Vol (\mathbb{S}^{n})
=
\frac{2 \pi^{\frac{n+1}{2}}}{\Gamma \left( \frac{n+1}{2} \right)}.
\end{equation}
 
 \begin{prop}
 \label{prop_6}
 Let $p \in \mathbb{S}^{2d}$, if a DPP of $N$ points on $\mathbb{S}^{2d}$ has associated reproducing kernel satisfying that $K(p,p)$ is constant, then
 \begin{equation*}
 K(p,p) 
 = 
\frac{N}{\text{Vol}(\mathbb{S}^{2d})} 
 = 
 \frac{N\Gamma\left( d+\frac{1}{2} \right)}{2\pi^{d+\frac{1}{2}}}.
 \end{equation*}
 \end{prop}
 
The proof follows the definition of first intensity function (see \cite[Definition 1.2.2.]{Hough_zerosof}).
Given a DPP of $N$ points with kernel $K(p,q)$ in $\mathbb{S}^{d}$ then the average number of points contained in the subset $A \subset \mathbb{S}^{d}$ is given by
\begin{equation*}
\int_{A} \rho_{1} dp
=
\int_{A} K(p,p) dp
=
\text{Vol}(A)K(p,p),
\end{equation*}
where $\rho_{1}$ is the first intensity joint function.
Note that if we have a homogeneus kernel then the expected number of points contained on any Borel subset depend only on its volume.
 

\section{An isotropic projection kernel in the $2$--sphere}
\label{sec:iran}


In \cite{EJP3733} Alishahi and Zamani study the energy of the spherical ensemble. A brief description of this point process is as follows: let $A,B$ be $N\times N$ matrices with complex Standard Gaussian entries, that is, each of the entries of $A$ and $B$ is independently and identically distributed by choosing its real and imaginary parts according to the real Gaussian distribution with mean $0$ and variance $1/2$. Then, the spherical ensemble is obtained by sending the generalized eigenvalues $\lambda_1,\ldots,\lambda_N\in\C$ of the matrix pencil $(A,B)$ to the sphere via the stereographic projection. Note that these $\lambda_i$ are the complex numbers $\lambda$ such that $\det(\lambda A-B)=0$ and there are (for generic $A,B$) $N$ solutions to this equation. Equivalently, one can search for the generalized eigenvalues of $(A,B)$ in the complex projective space, i.e. for points $(\alpha,\beta)\in\C^2$ such that $\det(\alpha B-\beta A)=0$, consider them as points in the Riemann sphere and send them to the unit sphere through an homotety. These two processes are equivalent. 

It has been shown by Krishnapur \cite{krishnapur2009} that this process is determinantal and the kernel comes from the projection kernel of the $N$-dimensional subspace of $\text{\rm L}^{2}(\mathbb{C},\C)$ with basis
\begin{equation*}
\left\{\sqrt{\frac{N}{\pi}\binom{N-1}{k}}\frac{z^{k}}{(1 + |z|^{2})^{\frac{N+1}{2}}} : 0 \leq k \leq N-1\right\},
\end{equation*}
where we are taking the usual Lebesgue measure $\mu$ in $\C$ which makes this basis orthonormal. \noindent The projection kernel is then given by
\begin{align*}
K(z,w) =&\sum_{k=0}^{N-1}\frac{N}{\pi}\binom{N-1}{k}\frac{(z\overline w)^k}{(1 + |z|^{2})^{\frac{N+1}{2}}(1 + |w|^{2})^{\frac{N+1}{2}}} \\
= & \frac{N}{\pi} \frac{(1 + z \overline{w})^{N-1}}{ (1 + |z|^{2})^{\frac{N+1}{2}} (1 + |w|^{2})^{\frac{N+1}{2}} }.
\end{align*}

\noindent The push--forward projection DPP in $\mathbb{S}^{2}$ given by the stereographical projection $\Pi: \mathbb{S}^{2} \longrightarrow \mathbb{C}$ (see \cite[Proposition 2.5]{BE2018}) has kernel:
\begin{equation*}
K_{*}^{(N)}(p,q) = K_{*}^{(N)}(\Pi^{-1}(z), \Pi^{-1}(w)) = \frac{N}{4\pi} \frac{(1+z\overline{w})^{N-1}}{(1 + |z|^{2})^{\frac{N-1}{2}} (1 + |w|^{2})^{\frac{N-1}{2}}}.
\end{equation*}
This point process has a number of nice properties, including $K_{*}^{(N)}(p,p)=N/(4\pi)$ (i.e. the process is homogeneus) and even more, it is isotropic:
\begin{equation*}
\left| K_{*}^{(N)}(p,q) \right| = \frac{N}{4\pi} \left( \frac{1 + \left\langle p,q \right\rangle }{2} \right)^{\frac{N-1}{2}}.
\end{equation*}
From this fact, the expected $s$--energy (as well as other interesting quantities) of random configurations drawn from this DPP can be computed, see \cite{EJP3733}.


\section{A  homogeneous projection kernel in the $2d$--sphere}
\label{sec:2d}


 It is not a trivial task to generalize the spherical ensemble to high--dimensional spheres. Here is the reason: the most natural approach is to take the subspace of $\mathrm{L}^2(\C^d,\C)$ spanned by the family
\begin{equation}\label{eq:01}
I_{d,L}=
\left\{
\sqrt{C_{\alpha_{1},...,\alpha_{d}}^{L}}
\frac{z_{1}^{\alpha_{1}}\ldots z_{d}^{\alpha_{d}}}{(1+\|z\|^{2})^{\frac{d+L+1}{2}}} :  {\alpha_{1}+\ldots +\alpha_{d} \leq L} 
\right\},
\end{equation}
where $C_{\alpha_{1},...,\alpha_{d}}^{L}$ is a constant that makes the basis orthonormal.
From \cite[Definition 3.1]{BE2018} we know that the reproducing kernel of the space spanned by $I_{d,L}$ is
\begin{equation}\label{eq:original}
K(z,w)
=
\frac{N d!}{\pi^{d}}
\frac{\left( 1+ \langle z,w \rangle \right)^{L}}{\left( 1 + ||z||^{2} \right)^{\frac{d+L+1}{2}} \left( 1 + ||w||^{2} \right)^{\frac{d+L+1}{2}}}.
\end{equation}

Then it is tempting to just map the associated reproducing kernel into the sphere $\mathbb{S}^{2d}$ by the stereographic projection. It turns out that the resulted associated DPP in $\S{2d}$ is not isotropic nor even homogeneus, and thus it does not satisfy the most basic properties required in a ``good'' kernel.

\noindent In order to avoid this problem, we will modify the corresponding DPP on $\mathbb{S}^{2d}$ by a weight function. Consider the stereographic projection from $\mathbb{S}^{2d}$ to $\mathbb{C}^{d}$:

\begin{equation}
\label{eq:stg}
\begin{matrix}
\Pi=\Pi_d:&\mathbb{S}^{2d}   &\longrightarrow & \C^d\equiv\mathbb{R}^{2d}\\
 &(p_1,\ldots,p_{2d+1}) &\rightarrow&\frac{1}{1-p_{2d+1}}(p_1,\ldots,p_{2d})\\
           &\left( \frac{2y_{1}}{\| y \|^{2}+1},\ldots ,\frac{2y_{2d}}{\| y \|^{2}+1},\frac{\| y \|^{2}-1}{ \|y\|^{2}+1} \right) &\leftarrow& (y_{1},\ldots y_{2d})\\
\end{matrix}
\end{equation}
where the identification $\C^d\equiv\R^{2d}$ is given by
\[
 (z_1,\ldots,z_d)\equiv\left(\mathrm{Re}(z_1),\mathrm{Im}(z_1),\ldots,\mathrm{Re}(z_d),\mathrm{Im}(z_d)\right).
\]

\begin{defn}
\label{def_1}
Let $ g: (0, \infty) \longrightarrow (0, \infty)$ be an increasing $\boldsymbol{\mathcal{C}^{1}}$ function  such that $\lim\limits_{x\rightarrow0}g(x) = 0$ and $\lim\limits_{x \rightarrow \infty} g(x) =  \infty$.
Now, we define the associated function 
\begin{equation*}
\begin{matrix}
\varphi=\varphi_g: & \mathbb{R}^{2d}\setminus\{0\} &\longrightarrow &\mathbb{R}^{2d}\setminus\{0\} \\
		 & x &\mapsto &g(\|x\|) \frac{x}{\|x\|}
\end{matrix}
\end{equation*}
Note that $\varphi$ is biyective and its inverse is given by $\varphi^{-1} (y) = \frac{g^{-1}(\|y\|)}{\|y\|} y$.
\end{defn}

Let $I_{d,L}$ be as in equation \eqref{eq:01}.
Then for all $N$ of the form $N={d+L \choose d}$ there exists a projection DPP of $N$ points  $\mathfrak{X}_\mathcal{H}$ in $\C^d$.
Let us consider the map $\phi=\Pi_d^{-1}\circ \varphi_g: \mathbb{C}^{d} \longrightarrow \mathbb{S}^{2d}$  for any fixed $g:(0,\infty)\to(0,\infty)$ as in Definition \ref{def_1}.
Then there exists a push--forward projection DPP in $\S{2d}$ (see \cite[Proposition 2.5]{BE2018}). We denote this process by $\mathfrak{X}_*^{(N,d,g)}$.

\begin{prop}
\label{prop_5}
Let $d\geq1$, let $g:(0,\infty)\to(0,\infty)$ be as in Definition \ref{def_1} and let $N$ be of the form $N={d+L \choose d}$. Then, $\mathfrak{X}_*^{(N,d,g)}$ is a DPP in $\mathbb{S}^{2d}$ with associated kernel
\begin{equation}
\label{eq_5}
K_*^{(N,d,g)}(p,q)=\frac{N d!  (1 + \left\langle z,w \right\rangle)^{L} R(\|z\|) R(\|w\|) }{{\pi}^{d}2^{2d}},
\end{equation}
where $z=\varphi_g^{-1}(\Pi_d(p)),\;w=\varphi_g^{-1}(\Pi_d(q))$ and
\begin{equation*}
R(t) = \frac{ (g(t)^{2}+1)^{d} t^{d-\frac{1}{2}}}{\sqrt{g'(t)} g(t)^{d-\frac{1}{2}} (1+t^{2})^{\frac{d+L+1}{2}}}.
\end{equation*}
\end{prop}

We now describe how to choose $g$ in such a way that the kernel $K_{*}^{(N,d,g)}(p,q)$ is homogeneus for all  $N={d+L \choose d}$.

\begin{prop}
\label{prop_3}
Let $d\geq1$. There exists a unique function $g=g_d$ satisfying the conditions of Definition \ref{def_1} that makes $K_{*}^{(N,d,g)}(p,p)$ constant for all $N$ of the form $N={d+L \choose d}$. Moreover, $g$ satisfies
\begin{equation*}
\mathrm{I}_{\frac{g^{2}}{1+g^2}}(d,d) =
\left(\frac{t^{2}}{1+t^{2}}\right)^{d},
\end{equation*}
where $\mathrm{I}_{\frac{g^{2}}{1+g^2}}(d,d)$ is the incomplete beta function (see equation \eqref{eq:betas} for a definition) and $g=g(t)$.
\end{prop}

We denote simply by $K_*^{(N)}(p,q)$ and $\mathfrak{X}_*^{(N)}$ the associated kernel and projection DPP, dropping in the notation the dependence on $d$.
Note that for $d=1$ we have
\begin{equation*}
\mathrm{I}_{\frac{g^{2}}{1+g^2}}(1,1) 
=
\frac{g(t)^{2}}{1+g(t)^2}
=
\frac{t^{2}}{1+t^{2}}
\Rightarrow
g(t) = t,
\end{equation*}
hence $\varphi_g:\R^2\to\R^2$ is the identity function and we recover the original spherical ensemble.
The graphic of $g$ for some values of $d$ can be see in Figure \ref{fig:g}.

\begin{figure}[htp]
\includegraphics[width=0.8\textwidth]{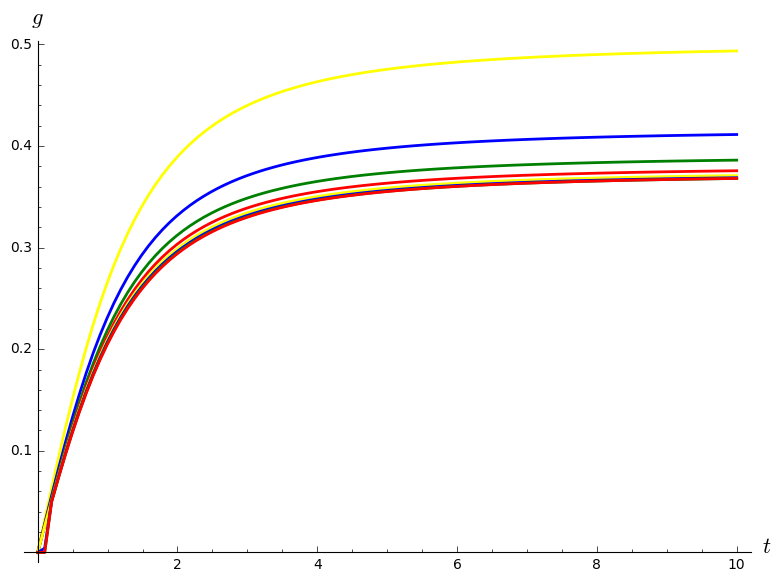}
\caption{The function $g$ for different values of $d$: from top to bottom, $d=2,3,4,5,6,7$ and $8$.}
\label{fig:g}
\end{figure}

\begin{thm}
\label{cor_3}
Let $d\geq1$ and let $N$ be of the form $N={d+L \choose d}$. Then, $\X$ is a homogeneous projection DPP in $\mathbb{S}^{2d}$ with kernel

\begin{equation*}
K_{*}^{(N)}(p,q) = \frac{N}{Vol(\S{2d})} \frac{\left( 1 + \left\langle z,w \right\rangle \right)^{L}}{(1 + \|z\|^{2})^{\frac{L}{2}} (1 + \|w\|^{2})^{\frac{L}{2}} }
\end{equation*}

\noindent where $z=\varphi_g^{-1}(\Pi_d(p)),\;w=\varphi_g^{-1}(\Pi_d(q))$.
\end{thm}


\section{The expected Riesz s-energy of the generalized spherical ensemble}
\label{section_energy}

The Riesz $s$--energy of a set on points $\omega_{N} = \{ x_{1}, \ldots  ,x_{N} \}$ on the sphere $\mathbb{S}^{d}$ with $s > 0$ is 
\begin{equation*}
\E{s}(\omega_{N}) = \displaystyle\sum_{i \neq j} \frac{1}{\| x_{i} - x_{j} \|^{s}}
\end{equation*}

An interesting problem regarding this energy consists on looking for the minimal value that the energy can reach for a set of $N$ points on a sphere of dimension $d$.
The asymptotic behavior (for $N\to\infty$) has been extensively studied.
In \cite{Brauchart, 10.2307/117605} it was proved that for $d > 2$ and $0<s<d$ there exist constants $c>C> 0$ (depending only on $d$ and $s$) such that
\begin{equation}\label{eq:bounds}
\min_{\omega_N}\left(\E{s}(\omega_{N})\right) 
\leq 
V_{s}(\mathbb{S}^{d})N^{2} - CN^{1 + \frac{s}{d}} + o(N^{1 + \frac{s}{d}}),
\end{equation}
\begin{equation*}
\min_{\omega_N}\left(\E{s}(\omega_{N})\right) 
\geq 
V_{s}(\mathbb{S}^{d})N^{2} - cN^{1 + \frac{s}{d}} + o(N^{1 + \frac{s}{d}}),
\end{equation*}
where $V_{s}(\mathbb{S}^{d})$ is the continuous s-energy for the normalized Lebesgue measure,
\begin{equation*}
V_{s}(\mathbb{S}^{d})
=
\frac{1}{Vol(\mathbb{S}^{d})^{2}} \displaystyle\int_{p,q \in \mathbb{S}^{d}} \frac{1}{\|p-q\|^{s}}dpdq
=
2^{d-s-1} \frac{\Gamma \left( \frac{d+1}{2} \right) \Gamma \left( \frac{d-s}{2} \right)}{\sqrt{\pi} \Gamma \left(  d- \frac{s}{2} \right)}.
\end{equation*}
Finding precise bounds for the constants in \eqref{eq:bounds} is an important open problem and has been addressed by several authors, see \cite{BHS2012,LB2015} for some very precise conjectures and \cite{Brauchart2015293} for a survey paper. 
A strategy for the upper bound is to take collections of random points in $\S{d}$ coming from DPPs and then compute the expected value of the energy (which is of course greater than or equal to the minimum possible value). 
This was done in the case of the harmonic ensemble, obtaining the best bounds known until date for general $s$ and even $d$.
Namely,
\begin{equation}\label{eq:BMOCliminf}
\begin{split}
\min_{\omega_N}\left(\E{s}(\omega_{N})\right) 
 \leq &
 V_{s}(\mathbb{S}^{d})N^{2} -
\frac{ 2^{s-\frac{s}{d}}V_s(\S{d}) d\,\Gamma\left( 1+\frac{d}{2} \right)   \Gamma\left( \frac{1+s}{2} 
\right)\Gamma\left(d-\frac{s}{2}\right) }
{ \sqrt{\pi} \Gamma\left( 1+\frac{s}{2} \right) \Gamma\left( 1+\frac{s+d}{2} 
\right)\left(d!\right)^{1-\frac{s}{d}}} N^{1 + \frac{s}{d}} \\
& + o(N^{1 + \frac{s}{d}}).
\end{split}
\end{equation}

In the case of odd-dimensional spheres, the best bound is obtained from a point process that is not determinantal but follows from a DPP (the projective ensemble described above), see \cite{BE2018}. For example the expression for the $2$-energy that one gets with this process is:
\begin{equation}\label{eq:BE}
\begin{split}
& \min_{\omega_N}\left(\E{2}(\omega_{N})\right) 
 \leq 
 V_{2}(\mathbb{S}^{d})N^{2} - \\
& \qquad
\frac{3^{1 - \frac{2}{2d+1}} (2d-1)^{1 - \frac{2}{2d+1}} (2d+1) \Gamma\left( d- \frac{1}{2} \right)^{2 - \frac{2}{2d+1}}}{2^{4 - \frac{2}{2d+1}} (d!)^{2 - \frac{4}{2d+1}}} N^{1 + \frac{2}{d}} + o(N^{1 + \frac{2}{d}}).
\end{split}
\end{equation}

\noindent For the generalized spherical ensemble, our result can be written as follows.
\begin{thm}
\label{cor_1}
Let $N$ be of the form $N = {d+L \choose d}$, $d \geq 1$ and $0 < s < 2d$, then

\begin{multline*}
\Esp{x\sim\X}{\E{s}(\omega_{N})}
\leq  
 V_{s}(\mathbb{S}^{2d})N^{2} \\
- 
\frac{Vol(\mathbb{S}^{2d-1}) \left( \frac{(2d-s) \left( 1 - \frac{1}{d} \right)}{2e} \right)^{d - \frac{s}{2}}}{(2d-s)Vol(\S{2d}))(d!)^{1-\frac{s}{2d}}}
 \left(1-\frac{e^{-1+\frac1{2d}}}{2\sqrt{1-\frac1{2d}}}\right)
 N^{1 + \frac{s}{2d}} 
+ o(N^{1 + \frac{s}{2d}}).
\end{multline*}

\end{thm}
If we compare our result with that from \eqref{eq:BMOCliminf}, we note that our bound is worst (see figure \ref{fig:Comp_har_sphe}). 
Nevertheless, the points coming from this generalized spherical ensemble do respect the asymptotic of the minimal energy, getting the correct exponent $1+s/d$ for the second term in the expansion.

\begin{figure}[htp]
\includegraphics[width=1\textwidth]{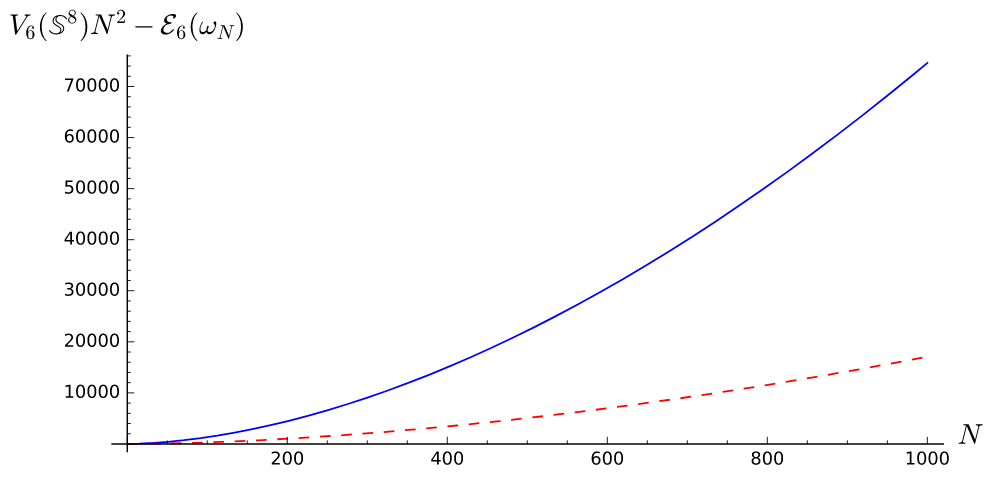}
\caption{Difference between the $6$-energy of the harmonic ensemble (blue solid line) and the $6$-energy of the generalized spherical ensemble (red dashed line) in $\mathbb{S}^{8}$. The harmonic ensemble shows a better behavior regarding energy asymptotics.}
\label{fig:Comp_har_sphe}
\end{figure}

The technical version of Theorem \ref{cor_1} is a bound of the expected value of Riesz s-energy for points drawn from $\X$.

\begin{thm}
\label{th_2}
Let $d\geq1$ and let $N$ be of the form $N = {d+L \choose d}$. Then, for $\tau > 0$ such that $\tau < 1-1/\sqrt{d}$ we have:

\begin{multline*}
 \Esp{x\sim\X}{\E{s}(\omega_{N})}
 \leq  N^{2} V_{s}(\mathbb{S}^{2d})- \frac{N^2 Vol(\mathbb{S}^{2d-1})}{(2d-s) Vol(\mathbb{S}^{2d})}\times  \\
 \left(1-\frac{\tau^2}{4}\right)^{d-1}
 \tau^{2d-s}
 \left(1- \tau^{2} \frac{(1+\tau^{2})^2 }{1 - \left(\frac1{\sqrt d}+ \tau\right)^2} \right)^{L}
 \left(1-\frac{e^{-1+\frac1{2d}}}{2\sqrt{1-\frac1{2d}}}\right)
\end{multline*}

\end{thm}

Theorem \ref{cor_1} will follow from Theorem \ref{th_2} by choosing the optimal value for $\tau$.


\section{An inequality regarding the incomplete Beta function}
\label{sec:beta}


In order to prove Theorem \ref{cor_1} we will need an inequality regarding Euler's incomplete Beta function. This inequality, which is sharp and might be of independent interest, is stated and proved in this section.
Let us recall the definition of Euler's incomplete Beta function (and its regularized version):
\begin{equation}\label{eq:betas}
 \mathrm{B}_x(a,b)=\int_0^xt^{a-1}(1-t)^{b-1}\,dt,\quad \mathrm{I}_x(a,b)=\frac{\mathrm{B}_x(a,b)}{\mathrm{B}(a,b)},
\end{equation}
which satisfies $\mathrm{I}_1(a,b)=1$.

\begin{thm}\label{th:ineqbeta}
Let $d\geq1$ and $s\in[0,1]$. Then,
\[
 d\mathrm{B}_{s}(d,d) \sqrt{1 - (\mathrm{I}_{s}(d,d))^{\frac{1}{d}}} \geq s^{d}(1-s)^{d},
 \]
 with an equality if and only if $s=0,1$.
\end{thm}
Before proving Theorem \ref{th:ineqbeta} we state a technical lemma:
\begin{lem}\label{lem:techbeta}
	The function 
	\[
	f(s)=dB_s(d,d)+s^d(1-s)^d\left(d-2ds-\sqrt{d^2(1-2s)^2+1+2d}\right)
	\]
	satisfies $f(s)\geq0$ for $s\in(0,1)$.
\end{lem}
\begin{proof}
	We note that $f(0)=0$ and we compute $f'(s)=ds^{d-1}(1-s)^{d-1}h(s)$ where
	\begin{multline*}
	h(s)=1+(1-2s)\left(d-2ds-\sqrt{d^2(1-2s)^2+1+2d}\right)+\\s(1-s)\left(-2+\frac{2d(1-2s)}{\sqrt{d^2(1-2s)^2+1+2d}}\right).
	\end{multline*}
	It suffices to prove that $h(s)\geq0$ for $s\in(0,1)$. Denoting $u(s)=d^2(1-2s)^2+1+2d$
	we have
	\[
	h(s)=1+d(1-2s)^2-2s(1-s)+(1-2s)\sqrt{u(s)}\left(\frac{2ds(1-s)-u(s)}{u(s)}\right).
	\]
	Since $2ds(1-s)<u(s)$ for all $s\in(0,1)$ we conclude that if $s>1/2$ then all the terms in $h(s)$ are positive and $h(s)\geq0$. It remains to prove $h(s)\geq0$ for $s\in(0,1/2)$. Then, the claim is equivalent to
	\[
	\left(1+d(1-2s)^2-2s(1-s)\right)^2\geq (1-2s)^2u(s)\left(\frac{2ds(1-s)-u(s)}{u(s)}\right)^2,
	\]
	that is
	\[
	u(s)\left(1+d(1-2s)^2-2s(1-s)\right)^2- (1-2s)^2(2ds(1-s)-u(s))^2\geq0.
	\]
	Expanding the terms, the last expression equals
	\[
	s^2(1  - s )^2 (4+8d)\geq0,
	\]
	which trivially holds, thus proving the lemma.
\end{proof}
\subsection{Proof of Theorem \ref{th:ineqbeta}}
Some elementary algebraic manipulations show that the inequality of the theorem is equivalent to:
\[
f(s)\geq \frac{1}{B(d,d)},\quad f(s)=Q(s)\,(1-Q(s)^2A(s))^d
\]
where $Q(s)=B_s(d,d)^{-1}$ and $A(s)=s^{2d}(1-s)^{2d}/d^2$. Now, note that
\[
f(1)=\frac{1}{B(d,d)},
\]
and hence it suffices to show that $f$ is a non--increasing function. We compute the derivative
\[
f'(s)=(1-Q(s)^2A(s))^{d-1}\left(Q'(s)-(1+2d)Q'(s)Q(s)^2A(s)-dQ(s)^3A'(s)\right),
\]
and hence it suffices to see that
\[
Q'(s)-(1+2d)Q'(s)Q(s)^2A(s)-dQ(s)^3A'(s)\leq0,\quad s\in(0,1),
\]
or equivalently we just have to see that
\[
dA'(s) B_s(d,d)+s^{d-1}(1-s)^{d-1}(B_s(d,d)^2-(1+2d)A(s))\geq0,\quad s\in(0,1).
\]
Computing the derivative $A'(s)$ and simplifying, this last inequality is equivalent to
\[
B_s(d,d)^2+2s^d(1-s)^d(1-2s)B_s(d,d)-\frac{1+2d}{d^2}s^{2d}(1-s)^{2d}\geq0,
\]
and hence also to
\begin{multline*}
\left(B_s(d,d)+s^d(1-s)^d(1-2s)\right)^2\geq s^{2d}(1-s)^{2d}(1-2s)^2+\frac{1+2d}{d^2}s^{2d}(1-s)^{2d}\\=
s^{2d}(1-s)^{2d}\left((1-2s)^2+\frac{1+2d}{d^2}\right),
\end{multline*}
which follows from Lemma \ref{lem:techbeta} after taking square roots. Theorem \ref{th:ineqbeta} now follows.


\section{Proofs of the main results}
\label{sec_proofs}


In order to prove the results presented in this paper, we will define two more functions.

\begin{defn}
\label{def_10}
Let $\theta:\mathbb{C}^{d}\to\mathbb{S}^{2d}$ be the mapping defined by $\theta (x) = \frac{(x,1)}{\|(x,1)\|}$.
We denote by $\phi(x) = (\theta \circ \varphi^{-1} \circ \Pi) (x)$, where $\varphi^{-1}$ and $\Pi$ are defined in Definition \ref{def_1}. 

Note that $\phi$ maps $\mathbb{S}^{2d}$ into its upper half.
\end{defn}

\subsection{Derivatives of the stereographic projection and other mappings}\label{sec:derivatives}

In this section we state an elementary lemma with the computation of the derivative of the stereographic projection and other mappings. 

Let $\mathbf{n}=(0,\ldots,0,1)\in\S{2d}$ be the north pole and let $p\in \mathbb{S}^{2d}$, $p\neq\pm\mathbf{n}$. It is useful to consider an orthonormal basis $\{\dot p^1,\ldots, \dot p^{2d-1},\dot p^{2d}\}$ of $p^\perp$ such that $\dot p^i\perp \mathbf{n}$ for $1\leq i\leq 2d-1$ and
\[
 \dot p^{2d}= \frac{\mathbf{n} -p_{2d+1} p}{ \sqrt{1-p_{2d+1}^{2}}}.
\]
\begin{lem}
\label{lem_3}

Let $\Pi$ be the stereographic projection \eqref{eq:stg}, let $\varphi_g$ be as defined in Definition \ref{def_1} and let $\theta$ be as defined in Definition \ref{def_10}. 
Then,

\begin{equation*}
\begin{split}
D\Pi(p)\dot{p} = &\left[ \frac{\dot{p_{1}}}{1-p_{2d+1} } + \frac{p_{1}\dot{p}_{2d+1}}{(1-p_{2d+1})^{2}},\ldots , \frac{\dot{p_{2d}}}{1-p_{2d+1} } + \frac{p_{2d}\dot{p}_{2d+1}}{(1-p_{2d+1})^{2}} \right], \\
D \varphi_g^{-1} (y) \dot{y} = &(g^{-1})'(\|y\|) \frac{y\Real{\left\langle y,\dot{y} \right\rangle}}{\| y \|^{2}}  + g^{-1}(\|y\|) \frac{\dot{y}\|y\| - y \frac{\Real{\left\langle y,\dot{y} \right\rangle}}{\| y \|}}{\|y\|^{2}}, \\
D\theta(x)\dot{x} = &\frac{(\dot{x},0)\|(x,1)\| - (x,1) \frac{\Real{\left\langle (x,1), (\dot{x},0)\right\rangle}}{\| (x,1) \|}}{\| (x,1) \|^{2}}.
\end{split}
\end{equation*}
Here we are denoting by $\Real{z}$ the real part of a complex number $z$.
\end{lem}

\begin{dem}
This computation is an exercise left to the reader. It is convenient to consider the basis described before the lemma.

$\square$
\end{dem}
\begin{cor}\label{cor:nj}
 The Jacobian determinants of  $\Pi$ and $\varphi_g^{-1}$ satisfy:
 \begin{align*}
  \mathrm{Jac}\Pi(p)=&\frac{1}{(1-p_{2d+1})^{2d}}=\left(\frac{1+\|\Pi(p)\|^2}{2}\right)^{2d},\\
  \mathrm{Jac}\varphi_g^{-1}(y)=& (g^{-1})'(\|y\|)\left(\frac{g^{-1}(\|y\|)}{\|y\|}\right)^{2d-1}.
 \end{align*}

\end{cor}
\begin{dem}
 Let $p\neq \pm\mathbf{n}$ and let $\dot p^1,\ldots,\dot p^{2d-1},\dot p^{2d}$ be the basis described at the beginning of this section. Then, it is clear from Lemma \ref{lem_3} that $D\Pi(p)$ preserves the orthogonality of the basis and a little algebra shows that it is an homothetic transformation with ratio $(1-p_{2d+1})^{-1}$, hence
 \[
  \mathrm{Jac}\Pi(p)=\frac{1}{(1-p_{2d+1})^{2d}}.
 \]
And noting that
 \[
  \|\Pi(p)\|^2=\frac{1+p_{2d+1}}{1-p_{2d+1}} \Rightarrow p_{2d+1}=\frac{\|\Pi(p)\|^2-1}{\|\Pi(p)\|^2+1},
 \]
 the first claim of the corollary follows.

For the second Jacobian, given $y\in\C^d\equiv\R^{2d}$ we consider an orthonormal basis $\dot v^1,\ldots,\dot v^{2d}$ whose last vector is $y/\|y\|$. Then, $D\varphi_g^{-1}$ preserves the orthogonality of this basis and hence we have
\[
\mathrm{Jac}\varphi_g^{-1}(y)=\|D\varphi_g^{-1}(y)(y/\|y\|)\|\cdot\prod_{i=1}^{2d-1}\|D\varphi_g^{-1}(y)\dot v^{i}\|,
\]
which from Lemma \ref{lem_3} equals
\[
 (g^{-1})'(\|y\|)\cdot\prod_{i=1}^{2d-1}\frac{g^{-1}(\|y\|)}{\|y\|},
\]
and the corollary follows.

\end{dem}

\subsection{Proof of Proposition \ref{prop_5}}

The push-forward of a projection DPP is a DPP, see \cite[Proposition 2.5]{BE2018} for a proof.
So $\mathfrak{X}_*^{(N,d,g)}$ is a DPP in $\S{2d}$, and from the same proposition and equation \eqref{eq:original}, we know that its associated kernel is 
\begin{equation}\label{aux}
 K_*^{(N,d,g)}(p,q)=\frac{N d!}{{\pi}^{d}}\frac{(1 + \left\langle z,w \right\rangle)^{L}\sqrt{|\text{\rm Jac}(\varphi_g^{-1}\circ \Pi)(z) \text{\rm Jac}(\varphi_g^{-1}\circ \Pi)(z)|}}{(1+\|z\|^{2})^{\frac{d+L+1}{2}} (1+\|w\|^{2})^{\frac{d+L+1}{2}}},
\end{equation}
where $z=\varphi_g^{-1}(\Pi(p)),\;w=\varphi_g^{-1}(\Pi(q))$. Now, from the chain rule,
\[
\mathrm{Jac}(\varphi_g^{-1}\circ \Pi)(p)=\mathrm{Jac}(\varphi_g^{-1})(\Pi(p))\mathrm{Jac}\Pi(p).
\]
From Corollary \ref{cor:nj}, this last equals
\[
 (g^{-1})'(\|\Pi(p)\|)\left(\frac{g^{-1}(\|\Pi(p)\|)}{\|\Pi(p)\|}\right)^{2d-1}\left(\frac{1+\|\Pi(p)\|^2}{2}\right)^{2d}.
\]
Now, $\|\Pi(p)\|=\|\varphi_g(z)\|=g(\|z\|)$ and thus we have:
\[
 \mathrm{Jac}(\phi^{-1})(p)= (g^{-1})'(g(\|z\|))\left(\frac{\|z\|}{g(\|z\|)}\right)^{2d-1}\left(\frac{1+g(\|z\|)^2}{2}\right)^{2d},
\]
namely:
\begin{equation}
 \label{eq:jacphiinv}
 \mathrm{Jac}(\phi^{-1})(p)= \frac{1}{g'(\|z\|)}\left(\frac{\|z\|}{g(\|z\|)}\right)^{2d-1}\left(\frac{1+g(\|z\|)^2}{2}\right)^{2d},
\end{equation}
and the same holds changing $p$ to $q$ and $z$ to $w$.
Putting together \eqref{aux} and \eqref{eq:jacphiinv} we get Proposition \ref{prop_5}.

\subsection{Proof of Proposition \ref{prop_3}}

During this proof, we denote by $K_*$ the kernel $K_*^{(N,d,g)}$. 
From Proposition \ref{prop_5}:
\begin{equation*}
\begin{split}
K_* (p,p)
= \frac{N d!}{\pi^{d}2^{2d}} \frac{  (g(\|z\|)^{2}+1)^{2d} \|z\|^{2d-1} }{ g'(\|z\|) g(\|z\|)^{2d-1}  (1+\|z\|^{2})^{d+1} }. \\
\end{split}
\end{equation*}
From Proposition \ref{prop_6}, if the DPP associated to $K_*$ is homogeneus then $K_*(p,p) = N/Vol(\S{2d})$ for all $p\in\S{2d}$. 
Thus, in order for the process to be homogeneus one must have:
\begin{equation*}
\frac{N d!}{\pi^{d}2^{2d}} \frac{  (g(t)^{2}+1)^{2d} t^{2d-1} }{ g'(t) g(t)^{2d-1}  (1+t^{2})^{d+1}} 
=\frac{N}{Vol(\S{2d})}=
\frac{N\Gamma\left( d +\frac{1}{2} \right)}{2\pi^{d +\frac{1}{2}}}, \quad t\in(0,\infty),
\end{equation*}
namely,
\begin{multline}\label{eq:gedo}
\frac{g'(t)g(t)^{2d-1}}{(g(t)^{2}+1)^{2d}}=\frac{\sqrt{\pi} d!}{2^{2d-1} \Gamma\left( d+\frac{1}{2} \right)}\frac{   t^{2d-1} }{    (1+t^{2})^{d+1} }=dB(d,d) \frac{  t^{2d-1} }{   (1+t^{2})^{d+1} },
\end{multline}
where we have used Legendre's duplication formula for the Gamma function (see \cite[Formula 5.5.5]{NIST:DLMF}),
\begin{equation}\label{eq:leg}
\Gamma(z) \; \Gamma\left(z + \frac{1}{2}\right) =
 2^{1-2z} \; \sqrt{\pi} \; \Gamma(2z) \,.\!
\end{equation}
\noindent Now we  integrate \eqref{eq:gedo} in both sides. On one hand,

\begin{equation*}
\begin{split}
\displaystyle\int dB(d,d) \frac{   t^{2d-1} }{   (1+t^{2})^{d+1} } d t
 = \frac{B(d,d)}{2} \frac{t^{2d}}{(1+t^{2})^{d}} \\
\end{split}.
\end{equation*}

\noindent On the other hand,

\begin{equation*}
\begin{split}
& \displaystyle\int \frac{ g^{2d-1}}{(g^{2}+1)^{2d}} d g
= 
\frac{1}{2} \left[B(d,d) - \mathrm{B}_{\frac{1}{1+g^2}}(d,d) \right] \end{split},
\end{equation*}
where $\mathrm{B}_x$ denotes Euler's incomplete Beta function (see equation \eqref{eq:betas} for a definition). We thus have (using the regularized incomplete Beta function):
\begin{equation*}
1 - \mathrm{I}_{\frac{1}{1+g^2}}(d,d)=
\frac{t^{2d}}{(1+t^{2})^{d}},
\end{equation*}
where the chosen integration constant is the unique with the property that $g(0)=0$. Recall that $\mathop{\mathrm{I}_{x}\/}\nolimits\!\left(a,b\right)=1-\mathop{\mathrm{I}_{1-x}\/}\nolimits\!%
\left(b,a\right)$ (see for example \cite{NIST:DLMF}). Hence, the function $g$ we are looking for must satisfy:
\begin{equation*}
\mathrm{I}_{\frac{g^{2}}{1+g^2}}(d,d) =
\left(\frac{t^{2}}{1+t^{2}}\right)^{d}.
\end{equation*}
Note that $g$ is well--defined since the regularized Beta function is bijective in the range $ [0,1]$ and for all $t \in (0,\infty)$, $0 < \frac{t^{2}}{1+t^{2}} < 1$.

\noindent We also check that $ g $ satisfies the claimed properties:
\begin{itemize}

\item $\boldsymbol{g}$ \textbf{is an increasing function}: $ g $ satisfies the differential equation \eqref{eq:gedo} and thus the derivative is positive.

\item $\boldsymbol{g(0) = 0}$ since we have chosen the correct integration constant.

\item $\boldsymbol{\lim\limits_{t \rightarrow \infty} g(t) =  \infty}$:  $\lim\limits_{t \rightarrow \infty} \mathrm{I}_{\frac{g^{2}}{1+g^2}}(d,d) = \lim\limits_{t \rightarrow \infty} \left(\frac{t^{2}}{1+t^{2}}\right)^{d} = 1$.

Now if $\lim\limits_{t \rightarrow \infty} \mathrm{I}_{\frac{g^{2}}{1+g^2}}(d,d) = 1 \Rightarrow \lim\limits_{t \rightarrow \infty} \frac{g^{2}}{1+g^2} = 1 \Rightarrow \lim\limits_{t \rightarrow \infty} g = \pm \infty $.
Since $g$ is increasing and $g(0) = 0$, the only possible solution is $\lim\limits_{t \rightarrow \infty} g = \infty.$

\item $\boldsymbol{g}$ \textbf{is} $\boldsymbol{\mathcal{C}^{\infty}}$: since $t\to\left(\frac{t^{2}}{1+t^{2}}\right)^{d}$ is $\mathcal{C}^{\infty}$ and so is the inverse regularized incomplete Beta function, we conclude that $t\to s(t)=\frac{g(t)^{2}}{1+g(t)^2}$ is $\mathcal{C}^{\infty}$. Then we can solve $g = \sqrt{\frac{s}{1-s}}$, so in the interval $s \in (0,1)$, $g$ is a composition of $\mathcal{C}^{\infty}$ functions whose denominator does not vanish and thereby is $\mathcal{C}^{\infty}$.

\end{itemize}

$\square$

\subsection{Proof of Theorem \ref{cor_3}}

\noindent We only have to replace the expresion for $g'(t)$ given in equation \eqref{eq:gedo} to compute $R(t)$ from Proposition \ref{prop_5}:
\[
R(t)=\frac{1}{\sqrt{d\mathrm{B}(d,d)}(1+t^2)^{L/2}}.
\]
We then have:
 \begin{equation*}
K_{*}^{(N)}(p,q) = \frac{Nd!}{\pi^d2^{2d}d\mathrm{B}(d,d)} \frac{\left( 1 + \left\langle z,w \right\rangle \right)^{L}}{(1 + \|z\|^{2})^{\frac{L}{2}} (1 + \|w\|^{2})^{\frac{L}{2}} }.
\end{equation*}

\noindent The expression for the constant $N/Vol(\S{2d})$ follows using \eqref{eq:leg} and \eqref{eq:volS}.

\subsection{Proof of Theorem \ref{th_2}}

It is well known (see for example \cite[Formula 1.2.2.]{Hough_zerosof}) that the expected value of the Riesz energy of a set of $N$ points coming from $\X$ satisfies:

\begin{equation*}
\begin{split}
\Esp{x\sim\X}{\E{s}(x_{1},\ldots ,x_{N})}  
& = \displaystyle\int_{\mathbb{S}^{2d} \times \mathbb{S}^{2d}} \frac{K_{*}^{(N)}(p,p)^{2} - |K_{*}^{(N)}(p,q)|^{2}}{\|p-q\|^{s}}  d p d q  \\
& = \frac{N^2}{Vol(\S{2d})^{2}} \displaystyle\int_{\mathbb{S}^{2d} \times \mathbb{S}^{2d}} \frac{1 - \left( \frac{|K_{*}^{(N)}(p,q)|}{K_{*}^{(N)}(p,p)} \right)^{2}}{\|p-q\|^{s}}  d p d q  \\
& = N^{2} V_{s}(\mathbb{S}^{2d})
- \frac{N^2}{Vol(\S{2d})^{2}}  \displaystyle\int_{\mathbb{S}^{2d} \times \mathbb{S}^{2d}} \frac{\left( \frac{|K_{*}^{(N)}(p,q)|}{K_{*}^{(N)}(p,p)} \right)^{2}}{\|p-q\|^{s}}  d p d q.
\end{split}
\end{equation*}

\noindent Bounding the integral in the last term is a non--trivial task. We will do it in several steps.

\begin{prop}
\label{prop_7}
Let $p,q$ be points of $\mathbb{S}^{2d}$, then
\begin{equation*}
\left( \frac{|K_{*}^{(N)}(p,q)|}{K_{*}^{(N)}(p,p)} \right)^{2} 
\geq
\max \left\lbrace  (1- \| \phi(p) - \phi(q) \|^{2})^{L} , 0\right\rbrace
\end{equation*}
\noindent where $\phi$ was defined in Definition \ref{def_10}. 
\end{prop}

\begin{dem}

It is obvious that $\left( \frac{|K_{*}^{(N)}(p,q)|}{K_{*}^{(N)}(p,p)} \right)^{2} \geq 0$. 
We will use the fact that for every unit vectors $x,y\in\C^a$, we have 
\[
|\left\langle x, y \right\rangle|^{2} \geq\Real{\langle x,y\rangle}^2= \left(1 - \frac{\|x-y\|^{2}}{2}\right)^2\geq1-\|x-y\|^2.
 \]
 Then,

\begin{equation*}
\begin{split}
\left( \frac{|K_{*}^{(N)}(p,q)|}{K_{*}^{(N)}(p,p)} \right)^{2} = & 
 \left( \frac{| 1 + \left\langle z,w \right\rangle |^{2}}{(1 + \|z\|^{2})(1 + \|w\|^{2})} \right)^{L} =
 \left| \frac{  \left\langle (z,1),(w,1) \right\rangle }{\sqrt{1 + \|z\|^{2}} \sqrt{1 + \|w\|^{2}}} \right|^{2L} = \\
& =  \left| \left\langle \frac{(z,1)}{\|(z,1)\|}, \frac{(w,1)}{\|(w,1)\|} \right\rangle\right|^{2L}
\geq  \left( 1 - \left\| \frac{(z,1)}{\|(z,1)\|} - \frac{(w,1)}{\|(w,1)\|} \right\|^{2} \right)^{L} = \\
& =  \left( 1 - \left\| \frac{(\varphi^{-1}(\Pi(p)),1)}{\|(\varphi^{-1}(\Pi(p)),1)\|} - \frac{(\varphi^{-1}(\Pi(q)),1)}{\|(\varphi^{-1}(\Pi(q)),1)\|} \right\|^{2} \right)^{L} = \\
& =  (1- \| \phi(p) - \phi(q) \|^{2})^{L}.
\end{split}
\end{equation*}

$\square$
\end{dem}

\begin{lem}
\label{lem_2}

Let $p,q \in \mathbb{S}^{2d}$, then the following inequality holds.

\begin{equation*}
\| \phi(p) - \phi(q) \| \leq \left(\|p-q\|+\|p-q\|^3\right) \sup \|D\phi(x)\|,
\end{equation*}
where the supremum is taken for $x$ in the geodesic from $p$ to $q$.
\end{lem}

\begin{dem}
Given two points $p, q \in \mathbb{S}^{2d}$, let $\alpha = d_{\mathbb{S}^{2d}} (p,q)$ where $d_{\mathbb{S}^{2d}}$ is the distance in the sphere, and let $\gamma$ be the geodesic segment from $p$ to $q$. Then we have 
\begin{align*}
\| \phi(p) - \phi(q) \| = &
\| \phi(\gamma(0)) - \phi(\gamma(\alpha)) \|
\leq \left\| \displaystyle\int_{0}^{\alpha} D\phi(\gamma(t)) \dot{\gamma}(t) dt \right\|\\
\leq&
 \displaystyle\int_{0}^{\alpha} \left\| D\phi(\gamma(t)) \dot{\gamma}(t) \right\| dt \leq \alpha\sup \|D\phi(x)\|.
\end{align*}
where $x$ lies on the geodesic from $p$ to $q$. 
We now note that
\[
 \alpha=d_{\mathbb{S}^{2d}} (p,q) = 2\arcsin\left( \frac{\| p-q \|}{2} \right)\leq \|p-q\|+\|p-q\|^3,
\]
the last since $2\arcsin(x/2)\leq x+x^3$ for $0\leq x\leq 2$ (a simple exercise left to the reader). The lemma follows.

$\square$
\end{dem}

\begin{prop}
\label{prop_8}
Let $\mathbf{n}=(0,\ldots,0,1)\in\S{2d}$ be the north pole and let $p\in \mathbb{S}^{2d}$, $p\neq\mathbf{n}$. Let $\dot p^1,\ldots,\dot p^{2d}$ be the orthonormal basis of $p^\perp$ defined in Section \ref{sec:derivatives}. Then, $D\phi(p)$ preserves the orthogonality of the basis and we have
\begin{equation*}
\left\| D\phi(p)\dot{p}^{i} \right\|=
\frac{g^{-1}\left( \sqrt{\frac{1+p_{2d+1}}{1-p_{2d+1}}} \right)}{\sqrt{1-p_{2d+1}^{2}}\sqrt{1 + \left( g^{-1}\left( \sqrt{\frac{1+p_{2d+1}}{1-p_{2d+1}}} \right) \right)^{2} }}, \quad 1\leq i\leq 2d-1,
\end{equation*}
and
\begin{equation*}
\left\| D\phi (p)\dot p^{2d}  \right\|=
\frac{(g^{-1})'\left( \sqrt{\frac{1+p_{2d+1}}{1-p_{2d+1}}} \right)}{\left(1 + g^{-1}\left( \sqrt{\frac{1+p_{2d+1}}{1-p_{2d+1}}} \right)^{2} \right) (1-p_{2d+1}) }.
\end{equation*}
In particular, $\|D\phi(p)\|$ is the supremum of these two quantities.
\end{prop}

\begin{dem}

\noindent We recall that $\phi(p) = (\theta \circ \varphi^{-1} \circ \Pi) (p)$.
 
\noindent Using the chain rule, 
\begin{equation*}
D\phi(p)\dot{p} = D(\theta \circ \varphi^{-1} \circ \Pi) (p)\dot{p} = D\theta (\varphi^{-1}(\Pi (p))) D\varphi^{-1}(\Pi(p))  D\Pi(p) \dot{p}.
\end{equation*}
From Lemma \ref{lem_3}, after some algebra we get for $1\leq i\leq 2d-1$:
\begin{equation*}
D\theta (\varphi^{-1}(\Pi (p))) D\varphi^{-1}(\Pi(p))D\Pi(p)\dot{p}^{i} = 
\frac{g^{-1}\left( \sqrt{\frac{1+p_{2d+1}}{1-p_{2d+1}}} \right)}{\sqrt{1-p_{2d+1}^{2}}\sqrt{1 + \left( g^{-1}\left( \sqrt{\frac{1+p_{2d+1}}{1-p_{2d+1}}} \right) \right)^{2} }} \dot{p}^{i},
\end{equation*}
while for $\dot p^{2d}$ we get
\begin{equation*}
\begin{split}
& D\theta (\varphi^{-1}(\Pi (p))) D\varphi^{-1}(\Pi(p))D\Pi(p)\frac{\mathbf{n} -p_{2d+1} p}{ \sqrt{1-p_{2d+1}^{2}}} = \\
& \frac{(g^{-1})'\left( \sqrt{\frac{1+p_{2d+1}}{1-p_{2d+1}}} \right)}{\sqrt{1-p_{2d+1}^{2}} (1-p_{2d+1})\sqrt{1 + \left( g^{-1}\left( \sqrt{\frac{1+p_{2d+1}}{1-p_{2d+1}}}\right) \right)^{2} }^{3} } 
\left(
\begin{array}{c}
p_{1} \\
\ldots  \\
p_{2d} \\
-g^{-1}\left( \sqrt{\frac{1+p_{2d+1}}{1-p_{2d+1}}} \right)\sqrt{1-p_{2d+1}^{2}}
\end{array}
\right)
\end{split}
\end{equation*}
Both the preservation of the orthogonality of the basis through $D\phi(p)$ and the formulas for the norm of $D\phi(p)\dot p^i$ follow, and the proposition is proved.

$\square$
\end{dem}

We need to compute the supremum among the two quantities of Proposition \ref{prop_8}, which is a nontrivial task. Following the same notation we have:

\begin{lem}
\label{lem_4}
Fix any $d\geq1$. The two following claims are equivalent:
\begin{enumerate}
	\item For all $p = (p_{1},\ldots ,p_{2d+1}) \in \mathbb{S}^{2d}$ we have
	\begin{equation*}
		\left\| D\phi(p)\dot{p}^{i} \right\| \geq \left\| D\phi (p) \frac{\mathbf{n} -p_{2d+1} p}{ \sqrt{1-p_{2d+1}^{2}}} \right\| .
	\end{equation*}
	(Here, $\dot p^i$, $1\leq i\leq 2d-1$, are as in Proposition \ref{prop_8}.)
	\item For all $s\in(0,1)$ we have
	\[
	d\mathrm{B}_{s}(d,d) \sqrt{1 - (\mathrm{I}_{s}(d,d))^{\frac{1}{d}}} \geq s^{d}(1-s)^{d}.
	\]
\end{enumerate}
\end{lem}

\begin{dem}
The claim is the result of a lengthy computation obtained by working on the expressions from Proposition \ref{prop_8} using the definition of $g$ given in Proposition \ref{prop_3}, which can be written as:
\begin{equation}\label{eq:otra}
 g^{-1}(s)=\left(\frac{\mathrm{I}_{\frac{s^2}{1+s^2}}(d,d)^{1/d}}{1-\mathrm{I}_{\frac{s^2}{1+s^2}}(d,d)^{1/d}}\right)^{1/2}.
\end{equation}

$\square$
\end{dem}

From Lemma \ref{lem_4} and Theorem \ref{th:ineqbeta} we readily have:

\begin{prop}
\label{prop_9}
For all $d \geq 1$ and all $p \in \mathbb{S}^{2d}$
\begin{equation*}
\left\| D\phi(p)\dot{p}^{i} \right\| 
\geq 
\left\| D\phi (p) \frac{\mathbf{n} -p_{2d+1} p}{ \sqrt{1-p_{2d+1}^{2}}} \right\| ,
\end{equation*}
where $\dot p^i$, $1\leq i\leq 2d-1$, are as in Proposition \ref{prop_8}.
\end{prop}

\begin{cor}
\label{cor_2}
Let $p,q \in \mathbb{S}^{2d}$, $\| p-q \| \leq \tau$ and $p_{2d+1} < \epsilon$ where $0<\epsilon+\tau<1$. Then, for every $x\in\S{2d}$ in the geodesic segment from $p$ to $q$ we have 
\begin{equation}
\label{eq_6}
\left\| D\phi(x)\right\|
\leq
\frac{g^{-1}\left( \sqrt{\frac{1+(\tau + \epsilon)}{1-(\tau + \epsilon)}} \right)}{\sqrt{1-(\tau + \epsilon)^{2}}\sqrt{1 + \left( g^{-1}\left( \sqrt{\frac{1+(\tau + \epsilon)}{1-(\tau + \epsilon)}} \right) \right)^{2} }} 
\leq 
\frac{1}{\sqrt{1-(\tau + \epsilon)^{2}}}
=
M_{\epsilon, \tau} .
\end{equation}
\end{cor}

\begin{dem}
\noindent The second inequality is trivial.
For the first one, note that from propositions \ref{prop_8} and \ref{prop_9}, for $x$ as in the hypotheses we have
\[
 \|D\phi(x)\dot{x}\|\leq\frac{g^{-1}\left( \sqrt{\frac{1+ x_{2d+1}}{1- x_{2d+1}}} \right)}{\sqrt{1- x_{2d+1}^{2}}\sqrt{1 + \left( g^{-1}\left( \sqrt{\frac{1+ x_{2d+1}}{1- x_{2d+1}}} \right) \right)^{2} }}\stackrel{\eqref{eq:otra}}{=}\sqrt{\frac{\mathrm{I}_{\frac{1+x_{2d+1}}{2}}^{1/d}}{1-x_{2d+1}^2}}
\]
It is clear that this is an increasing function of $x_{2d+1}$. The claim of the proposition follows noting that $p_{2d+1}\leq \epsilon$ and $\|p-q\|\leq \tau$ implies $|x_{2d-1}-p_{2d+1}|\leq \tau$ and hence $x_{2d+1}\leq \tau+\epsilon$.

$\square$
\end{dem}

\begin{prop}
\label{prop_10}
Let $\epsilon,\tau \in \left( 0, 1\right)$ with $\epsilon+\tau<1$. Then for all $p,q \in \mathbb{S}^{2d}$ we have
\begin{multline*}
\displaystyle\int_{p,q \in \mathbb{S}^{2d}} \frac{\left( \frac{|K_{*}^{(N)}(p,q)|}{K_{*}^{(N)}(p,p)} \right)^{2}}{\|p-q\|^{s}}  d p d q  
\geq\\
\frac{
W(\epsilon) 
Vol(\mathbb{S}^{2d-1})}{
(2d-s)} \left(1-\frac{\tau^2}{4}\right)^{d-1}
\tau^{2d-s}
\left(1- \tau^{2} \frac{(1+\tau^{2})^2 }{1 - (\epsilon + \tau)^2} \right)^{L}
\end{multline*}
where $W(\epsilon)$ is the volume of the set of $p\in\S{2d}$ such that $p_{2d+1}\leq\epsilon$. 
\end{prop}

\noindent In order to prove Proposition \ref{prop_10}, we present Lemma \ref{prop_2} (that follows from the change of variables formula applied to the projection from the cylinder to the sphere).

\begin{lem}
\label{prop_2}
If $f:\S{m}\to\R$ satisfies $f(q) = g(\langle p,q\rangle)$, for some $p\in\S{m}$ and some $g:\left[-1,1\right] \longrightarrow \mathbb{R}$, then
\begin{equation*}
\displaystyle\int_{p \in \mathbb{S}^{m}} f(p) dp = 
Vol(\mathbb{S}^{m-1}) \displaystyle\int_{-1}^{1} g(t) (1-t^{2})^{\frac{m}{2} - 1},
\end{equation*}
assuming that $f$ is integrable or $f$ is measurable and non--negative.
\end{lem}

\begin{dem}[Proof of Proposition \ref{prop_10}]
\noindent From lemmas \ref{prop_7} and \ref{lem_2} and Corollary \ref{cor_2}, $\tau, \epsilon > 0$ and $0 < \tau + \epsilon < 1$ we have:
\begin{multline*}
\displaystyle\int_{\mathbb{S}^{2d} \times \mathbb{S}^{2d}} \frac{\left( \frac{|K_{*}^{(N)}(p,q)|}{K_{*}^{(N)}(p,p)} \right)^{2}}{\|p-q\|^{s}}  d p d q\geq \displaystyle\int_{\| p-q \|\leq \tau, p_{2d+1} \leq \epsilon} \frac{\left( \frac{|K_{*}^{(N)}(p,q)|}{K_{*}^{(N)}(p,p)} \right)^{2}}{\|p-q\|^{s}}  d p d q
\geq\\
\displaystyle\int_{\| p-q \|\leq \tau, p_{2d+1} \leq \epsilon} \frac{ \left(1- \frac{(1+\tau^{2})^2 \|p-q\|^2 }{1 - (\epsilon + \tau)^2}\right)^{L}}{\|p-q\|^{s}}  d p d q. \\
\end{multline*}

We apply Fubini's theorem and Lemma \ref{prop_2}, transforming the last integral into: 
\begin{equation*}
\begin{split}
& \displaystyle\int_{p_{2d+1}  \leq \epsilon} 
\left[
\displaystyle\int_{\| p-q \|\leq \tau} 
\frac{ \left(1- \frac{(1+\tau^{2})^2 \|p-q\|^2 }{1 - (\epsilon + \tau)^2} \right)^{L}}{\|p-q\|^{s}}  d q 
\right] d p = \\
\end{split}
\end{equation*}
\begin{equation*}
\begin{split}
& = \displaystyle\int_{p_{2d+1}  \leq \epsilon} 
\left[
\displaystyle\int_{\sqrt{2-2\langle p,q\rangle}\leq \tau} 
\frac{ \left(1- \frac{(1+\tau^{2})^2 }{1 - (\epsilon + \tau)^2} (2-2\langle p,q\rangle)\right)^{L}}{\sqrt{2-2\langle p,q\rangle}^{\,s}}  d q 
\right] d p = \\
\end{split}
\end{equation*}
\begin{equation}
\label{eq_7}
\begin{split}
& = Vol(\mathbb{S}^{2d-1}) \displaystyle\int_{1 - \frac{\tau^2}{2}}^{1} 
\frac{ \left(1- \frac{(1+\tau^{2})^2 }{1 - (\epsilon + \tau)^2} (2-2t)\right)^{L}}{\sqrt{2-2t}^{\,s}} (1-t^2)^{d-1} d t 
\displaystyle\int_{p_{2d+1}  \leq \epsilon}  d p.
\end{split}
\end{equation}

\noindent Then, \eqref{eq_7} equals:

\begin{equation*}
\frac{W(\epsilon) Vol(\mathbb{S}^{2d-1})}{2^\frac{s}{2}} 
\displaystyle\int_{1 - \frac{\tau^2}{2}}^{1} 
\frac{ \left(1- \frac{(1+\tau^{2})^2}{1 - (\epsilon + \tau)^2} (2-2t)\right)^{L}}{\sqrt{1-t}^{s}} (1-t^2)^{d-1} d t.
\end{equation*}

\noindent where $W(\epsilon)$ is the volume of the set of points of $\mathbb{S}^{2d}$ such that their last coordinate is less or equal to $\epsilon$.
With the change of variables $u=1-t$ and using $1-t^2=(1-t)(1+t)$ we have proved that:
\begin{multline}\label{eq_10}
\displaystyle\int_{\mathbb{S}^{2d} \times \mathbb{S}^{2d}} \frac{\left( \frac{|K_{*}^{(n)}(p,q)|}{K_{*}^{(n)}(p,p)} \right)^{2}}{\|p-q\|^{s}}  d p d q\geq\\ \frac{W(\epsilon) Vol(\mathbb{S}^{2d-1})}{2^\frac{s}{2}} \left(2-\frac{\tau^2}{2}\right)^{d-1}
\displaystyle\int_{0}^{\frac{\tau^2}{2}} 
 \left(1- 2 \frac{(1+\tau^{2})^2 }{1 - (\epsilon + \tau)^2} u\right)^{L} u^{d-1-\frac{s}{2}} d u.
\end{multline}

\noindent Since $u \leq \frac{\tau^{2}}{2}$, $\left(1- 2 \frac{(1+\tau^{2})^2 }{1 - (\epsilon + \tau)^2} u\right)^{L} \geq \left(1- \tau^{2} \frac{(1+\tau^{2})^2 }{1 - (\epsilon + \tau)^2} \right)^{L}$ and so

\begin{multline}\label{eq:22}
\displaystyle\int_{\mathbb{S}^{2d} \times \mathbb{S}^{2d}} \frac{\left( \frac{|K_{*}^{(n)}(p,q)|}{K_{*}^{(n)}(p,p)} \right)^{2}}{\|p-q\|^{s}}  d p d q\geq\\ 
\frac{W(\epsilon) Vol(\mathbb{S}^{2d-1})}{2^\frac{s}{2}} \left(2-\frac{\tau^2}{2}\right)^{d-1}
\left(1- \tau^{2} \frac{(1+\tau^{2})^2 }{1 - (\epsilon + \tau)^2} \right)^{L} 
\displaystyle\int_{0}^{\frac{\tau^2}{2}} 
 u^{d-1-\frac{s}{2}} d u.
\end{multline}

\noindent We then have proved the following lower bound for the integral in the proposition:
\[
\frac{
W(\epsilon) 
Vol(\mathbb{S}^{2d-1})}{
(2d-s)} \left(1-\frac{\tau^2}{4}\right)^{d-1}
\tau^{2d-s}
\left(1- \tau^{2} \frac{(1+\tau^{2})^2 }{1 - (\epsilon + \tau)^2} \right)^{L}.
\]
The proof is now complete.

$\square$
\end{dem}

\subsubsection{Proof of Theorem \ref{th_2}}
First we are going to give a bound for $W(\epsilon)$.

\begin{prop}\label{prop:enesima}
Let $r>0$ and let $\vartheta\left( \frac{\pi}{2} + r \right)$ be the volume of the spherical cap of radius $\frac{\pi}{2} + r$ in $\mathbb{S}^{n+1}$, then 
\begin{equation*}
\vartheta\left( \frac{\pi}{2} + r \right) 
\geq
Vol(\mathbb{S}^{n+1}) \left( 1 - \frac{e^{\frac{-r^2 n}{2}}}{2} \sqrt{1 + \frac{1}{n}} \right)
\end{equation*}
\end{prop}

\begin{dem}
As in \cite[Corollary 2.2]{milman1986asymptotic} we consider the normalized measure of $\vartheta\left( \frac{\pi}{2} + r \right)$,
\begin{equation*}
\frac{\vartheta\left( \frac{\pi}{2} + r \right)}{Vol(\mathbb{S}^{n+1})}
=
\frac{\int_{\frac{-\pi}{2}}^{r} \cos^{n}\theta d \theta}{\int_{\frac{-\pi}{2}}^{\frac{\pi}{2}} \cos^{n}\theta d \theta}
\end{equation*}

\noindent The same result shows that

\begin{equation*}
1 - \frac{\vartheta\left( \frac{\pi}{2} + r \right)}{Vol(\mathbb{S}^{n+1})}
\leq
\frac{e^{\frac{-r^2 n}{2}} \sqrt{\frac{\pi}{2}}}{2 \sqrt{n} I_{n}},
\end{equation*}

\noindent where $I_{n} = \int_{0}^{\frac{\pi}{2}} \cos^{n}\theta d \theta = \frac{\sqrt{\pi} \Gamma \left( \frac{n}{2} + \frac{1}{2} \right)}{2 \Gamma \left( \frac{n}{2} + 1 \right)}$.
Applying Gautschi's inequality (see \cite[Theorem 1]{LM2016}) we obtain that $\frac{\Gamma \left( \frac{n}{2} + \frac{1}{2}  \right)}{\Gamma \left( \frac{n}{2} + 1\right)} \geq \sqrt{\frac{2}{n + 1}}$ so 

\begin{equation*}
1 - \frac{\vartheta\left( \frac{\pi}{2} + r \right)}{Vol(\mathbb{S}^{n+1})}
\leq
\frac{e^{\frac{-r^2 n}{2}}}{2} \sqrt{1 + \frac{1}{n}} 
\end{equation*}

\noindent and Proposition \ref{prop:enesima} is proved.

$\square$
\end{dem}

Now, taking $n=2d-1$ in Proposition \ref{prop:enesima},

\begin{equation*}
W \left( \frac{1}{\sqrt{d}} \right)
 \geq\vartheta\left( \frac{\pi}{2} + \frac{1}{\sqrt{d}} \right)
\geq
Vol(\mathbb{S}^{2d})
\left(1-\frac{e^{-1+\frac1{2d}}}{2\sqrt{1-\frac1{2d}}}\right)
\end{equation*}

\noindent and now we can substitute in the formula from Proposition \ref{prop_10} obtaining

\begin{multline*}
 \Esp{x\sim\X}{\E{s}(\omega_{N})}
\leq  N^{2} V_{s}(\mathbb{S}^{2d})- \frac{N^2 Vol(\mathbb{S}^{2d-1})}{(2d-s) Vol(\mathbb{S}^{2d})}\times  \\
\left(1-\frac{\tau^2}{4}\right)^{d-1}
\tau^{2d-s}
\left(1- \tau^{2} \frac{(1+\tau^{2})^2 }{1 - \left(\frac1{\sqrt d}+ \tau\right)^2} \right)^{L}
\left(1-\frac{e^{-1+\frac1{2d}}}{2\sqrt{1-\frac1{2d}}}\right)
\end{multline*}

\noindent where $0 < \tau  < 1-1/\sqrt{d} $. This finishes the proof of Theorem \ref{th_2}.

$\square$

\subsection{Proof of Theorem \ref{cor_1}}\label{sec:proofmainintro}

From Theorem \ref{th_2},
\begin{multline*}
 \frac{\Esp{x\sim\X}{\E{s}(\omega_{N})}
- N^{2} V_{s}(\mathbb{S}^{2d})}{N^{1+\frac{s}{2d}}}
\leq -\frac{N^{1-\frac{s}{2d}} Vol(\mathbb{S}^{2d-1})}{(2d-s) Vol(\mathbb{S}^{2d})}\times \\
\left(1-\frac{\tau^2}{4}\right)^{d-1}
\tau^{2d-s}
\left(1- \tau^{2} \frac{(1+\tau^{2})^2 }{1 - \left(\frac1{\sqrt{d}}+ \tau\right)^2} \right)^{L}
 \left(1-\frac{e^{-1+\frac1{2d}}}{2\sqrt{1-\frac1{2d}}}\right)
\end{multline*}
Fix any $C>0$ and let $\tau = \sqrt{C/L}$ (which satisfies $\tau<1-\frac1{\sqrt{d}}$ for large enough $L$). Then, the expression above equals
\[
\frac{-N^{1-\frac{s}{2d}} Vol(\mathbb{S}^{2d-1})}{Vol(\mathbb{S}^{2d}) (2d-s) L^{d - \frac{s}{2}}}
\frac{C^{d - \frac{s}{2}}}{e^\frac{dC}{d-1}}
 \left(1-\frac{e^{-1+\frac1{2d}}}{2\sqrt{1-\frac1{2d}}}\right)
Q_L,
\]
where $Q_L$ is a sequence with $\lim_{L\to\infty}Q_L=1$.

\noindent We recall that $N = {d+L \choose d}$, which implies
\[
 \frac{L^d}{d!}\leq\frac{(L+d)\cdots (L+1)}{d!}=\binom{d+L}{d}=N.
\]
We then have proved:
\begin{multline*}
  \frac{\Esp{x\sim\X}{\E{s}(\omega_{N})}
- N^{2} V_{s}(\mathbb{S}^{2d})}{N^{1+\frac{s}{2d}}}
\leq \\
 -\frac{N^{1-\frac{s}{2d}} Vol(\mathbb{S}^{2d-1})}{Vol(\mathbb{S}^{2d}) (2d-s) (Nd!)^{1-\frac{s}{2d}}}
\frac{C^{d - \frac{s}{2}}}{e^\frac{dC}{d-1}}
 \left(1-\frac{e^{-1+\frac1{2d}}}{2\sqrt{1-\frac1{2d}}}\right)
Q_L,
\end{multline*}
which is valid for all $C>0$. The optimal $C$ is easily computed:
\[
C=d-1-\frac{d-1}{2d}s.
\]
The theorem follows substituting this value of $C$ in the formula above.

$\square$



\end{document}